\documentclass[oneside,british]{amsart}
\usepackage[T1]{fontenc}
\usepackage[utf8]{inputenc}
\usepackage{mathrsfs}
\usepackage{amstext}
\usepackage{amsthm}
\usepackage{amssymb}

\makeatletter
\numberwithin{equation}{section}
\numberwithin{figure}{section}
\theoremstyle{plain}
\newtheorem{thm}{\protect\theoremname}[section]
\theoremstyle{plain}
\newtheorem{lem}[thm]{\protect\lemmaname}

\makeatother

\usepackage{babel}
\providecommand{\lemmaname}{Lemma}
\providecommand{\theoremname}{Theorem}

\begin{document}
\global\long\def\I{\mathbb{I}}%
\global\long\def\p{\mathbb{P}}%
\global\long\def\E{\mathbb{E}}%
 
\address{University of Hamburg, Hamburg, Germany}
\email{leylajafarova08@gmail.com}
\thanks{The author is grateful to Vitalii Konarovskyi for his valuable guidance
and helpful discussions.}
\title{Sticky Brownian Motion with a Locally Finite Atomic Sticky Measure}
\author{L.A. Jafarova}
\begin{abstract}
We study sticky Brownian motion on $\mathbb{R}$ with a set of sticky
points $A$. In contrast to the classical case of a single sticky
point, the set $A$ may have a more complicated structure and, in
particular, may have accumulation points. The stickiness at the points
of $A$ is described by a locally finite purely atomic measure $\nu$.

We extend the classical time-change approach for sticky Brownian motion
to this setting and use it to construct a weak solution of the corresponding
stochastic differential equation with a local-time condition. We then
prove uniqueness in law of the solution. 
\end{abstract}

\maketitle

\section{Introduction}

Boundary behaviour is fundamental in the study of diffusion processes.
Feller \cite{key-8} (see also \cite{key-9,key-10}) introduced sticky
boundaries while characterizing infinitesimal generators of strong
Markov processes on $[0,\infty)$ that behave as Brownian motion on
$(0,\infty)$. Using the Hille--Yosida approach, he obtained the
boundary condition 
\begin{equation}
f'(0+)=\frac{1}{2\mu}f''(0+),\qquad\mu\in(0,\infty),\label{eq:-12}
\end{equation}
which characterizes a sticky boundary. A probabilistic construction
was later provided by Itô and McKean \cite{key-11}, who constructed
sticky Brownian motion as a time change of reflecting Brownian motion.
More precisely, for a standard Brownian motion $B$, define the additive
functional 
\[
A_{t}=t+\frac{1}{\mu}L_{t}^{0}(B)
\]
and its inverse 
\[
T_{t}:=A_{t}^{-1},\qquad t\geq0,
\]
where $L^{0}(B)$ denotes the local time of $B$ at $0$. Then the
sticky process is given by 
\[
X_{t}=|B_{T_{t}}|.
\]

The local-time term increases the time-change functional whenever
the process visits the boundary, causing the inverse clock to slow
down and hence producing a positive occupation time at the boundary.

The Itô--McKean time-change construction can also be used for a sticky
point located in the interior of the state space. In this case one
starts with a standard Brownian motion $B$ instead of reflecting
Brownian motion. For a sticky point at $0$, the time change is based
on the additive functional
\[
A_{t}=t+\frac{1}{\mu}L_{t}^{0}(B),
\]
and the time-changed process is given by 
\[
X_{t}=B_{T_{t}},
\]
where $T$ is the inverse of $A$. The corresponding condition for
the generator can be written as
\begin{equation}
f'(0+)-f'(0-)=\frac{1}{\mu}f''(0\pm).\label{eq:-13}
\end{equation}

A systematic treatment of stochastic differential equations with sticky
behaviour was developed by Engelbert and Peskir \cite{key-7}. Focusing
on Brownian motion on $\mathbb{R}$ sticky at the origin, they formulated
the process through the system 
\begin{align*}
 & dX_{t}=\I_{\{X_{t}\neq0\}}\,dB_{t},\\
 & \I_{\{X_{t}=0\}}dt=\frac{1}{\mu}dL_{t}^{0}(X),
\end{align*}
where $X_{0}=x\in\mathbb{R}$ and $\mu\in(0,\infty)$ is a fixed constant.
Building on the Itô--McKean time-change construction, they established
existence and uniqueness in law for solutions of this system, providing
an SDE formulation of sticky Brownian motion.

A generalization of equations of this type was studied by Fattler,
Grothaus, and Steil in \cite{key-6}, where they investigated multidimensional
sticky diffusions using Dirichlet form methods. Under suitable assumptions
on a function $\rho$, they established existence and uniqueness in
law of solutions to 
\begin{align}
 & dX_{t}^{i}=\I_{\mathbb{R}^{d}\setminus A}(X_{t})\sqrt{2}\,dB_{t}^{i}+\I_{\mathbb{R}^{d}\setminus A}(X_{t})\partial_{x_{i}}\log\rho(X_{t})\,dt,\label{eq:-46}\\
 & X_{0}^{i}=x_{i},\nonumber 
\end{align}
for $i\in\{1,\ldots,d\}$ with given invariant distribution, where
$(B_{t}^{i})_{t\geq0}$ are independent one-dimensional standard Brownian
motions and $A$ is a closed set of $d$-dimensional Lebesgue measure
zero. In the one-dimensional case, the authors considered a countable
set $A$ such that every compact subset of $\mathbb{R}$ intersects
$A$ in finitely many points.

Motivated by these developments, we consider sticky Brownian motion
on $\mathbb{R}$ with a more general set $A$ of sticky points. To
each $c\in A$ we associate a constant $\mu_{c}>0$ and consider the
purely atomic measure 
\[
\nu=\sum_{c\in A}\frac{1}{\mu_{c}}\delta_{c}.
\]
Our main assumption is that $\nu$ is locally finite, that is, 
\[
\sum_{c\in A\cap K}\frac{1}{\mu_{c}}<\infty
\]
for every compact set $K\subset\mathbb{R}$. In particular, the set
$A$ itself does not have to be locally finite and may have accumulation
points.

The measure $\nu$ appears naturally in the Ito-McKean time-change
construction. For a Brownian motion $B$, the corresponding local-time
additive functional is 
\[
\int_{\mathbb{R}}L_{t}^{x}(B)\nu(dx)=\sum_{c\in A}\frac{1}{\mu_{c}}L_{t}^{c}(B).
\]
Since $\nu$ is locally finite and $x\mapsto L_{t}^{x}(B)$ has compact
support for every fixed $t$, this integral is well defined.

It is useful to compare the atomic case with the case where $\nu$
is absolutely continuous with respect to Lebesgue measure. If 
\[
\nu(dx)=\rho(x)dx,
\]
then the occupation-time formula gives 
\[
\int_{\mathbb{R}}L_{t}^{x}(B)\nu(dx)=\int_{0}^{t}\rho(B_{s})\,ds.
\]
Thus, in the absolutely continuous case the local-time functional
reduces to an ordinary time integral, and the corresponding time change
leads to a diffusion described by a classical SDE. In the purely atomic
case considered here, the local-time terms remain in the time change
and produce sticky behaviour at the points of $A$.

We study the following system: 
\begin{align}
 & dX_{t}=\I_{\{X_{t}\notin A\}}\,dB_{t},\nonumber \\
 & X_{0}=x_{0},\label{eq:-3}\\
 & \frac{1}{\mu_{c}}L_{t}^{c}(X)=\int_{0}^{t}\I_{\{X_{s}=c\}}\,ds,\qquad c\in A.\nonumber 
\end{align}
Here $L_{t}^{c}(X)$ denotes the local time of $X$ at the point $c$.
The first equation describes the Brownian dynamics away from the set
$A$, while the second equation describes the time spent by the process
at each sticky point.

A weak solution of (\ref{eq:-3}) is a pair $(X,B)$ of adapted continuous
processes defined on a filtered probability space $(\Omega,\mathscr{F},\mathbb{F},\mathbb{P})$,
where $B$ is an $\mathbb{F}$-Brownian motion, such that
\begin{align*}
 & X_{t}=x_{0}+\int_{0}^{t}\I_{\{X_{s}\notin A\}}dB_{s},\\
 & \frac{1}{\mu_{c}}L_{t}^{c}(X)=\int_{0}^{t}\I_{\{X_{s}=c\}}ds,\qquad c\in A.
\end{align*}

The main result of this work is the following theorem.
\begin{thm}
\label{thm:main}Let $A\subset\mathbb{R}$ be countable and let $(\mu_{c})_{c\in A}$
be a family of positive constants. Assume that the measure
\[
\nu=\sum_{c\in A}\frac{1}{\mu_{c}}\delta_{c}
\]
is locally finite. Then, for every $x_{0}\in\mathbb{R}$, there exists
a weak solution to (\ref{eq:-3}). Moreover, the solution is unique
in law.
\end{thm}

The existence part is proved by extending the Itô-McKean time-change
construction. More precisely, we introduce the additive functional
\[
R_{t}=t+\int_{\mathbb{R}}L_{t}^{c}(\tilde{W})\nu(dc)=t+\sum_{c\in A}\frac{1}{\mu_{c}}L_{t}^{c}(\tilde{W})
\]
and construct the solution by applying the inverse time change to
a Brownian motion $\tilde{W}$. Uniqueness in law is proved by reversing
this construction for an arbitrary weak solution.

In Section \ref{sec:dirichlet-identification}, we also establish
a connection between the process constructed here and the one constructed
in \cite{key-6} using Dirichlet-form approach. 

\section{Proof of the main result}

In this section we prove weak existence and uniqueness in law for
(\ref{eq:-3}). The proof follows the time-change approach of Engelbert
and Peskir \cite{key-7}, adapted to the present setting with a countable
set of sticky points.

\subsection{Weak existence \label{subsec:Formulas-for-local}}

We first prove the existence of a weak solution to (\ref{eq:-3}).
Let $\widetilde{B}$ be a standard Brownian motion defined on a probability
space $(\widetilde{\Omega},\widetilde{\mathscr{F}},\widetilde{\mathbb{P}})$,
and let $\widetilde{\mathbb{F}}$ denote its natural filtration. Set
\[
\widetilde{W}_{t}=x_{0}+\widetilde{B}_{t},\quad t\geq0.
\]
Consider the continuous increasing functional 
\[
R_{t}=t+\sum_{c\in A}\frac{1}{\mu_{c}}L_{t}^{c}(\widetilde{W})=t+\int_{\mathbb{R}}L_{t}^{c}(\widetilde{W})\nu(dc),
\]
where 
\[
\nu=\sum_{c\in A}\frac{1}{\mu_{c}}\delta_{c}.
\]
Since $\nu$ is locally finite and $c\mapsto L_{t}^{c}(\widetilde{W})$
has compact support for every fixed $t$, the integral is finite.
Moreover, $R$ is continuous and strictly increasing, and 
\[
R_{t}\geq t,
\]
so that $R_{t}\to\infty$ as $t\to\infty$.

Therefore its inverse 
\[
T_{t}:=R_{t}^{-1},\quad t\geq0,
\]
is well defined. We define the time-changed process 
\begin{equation}
Y_{t}=\widetilde{W}_{T_{t}}=x_{0}+\widetilde{B}_{T_{t}},\qquad t\geq0.\label{eq:-26}
\end{equation}

We first prove a lemma that will be used below.
\begin{lem}
\label{lem:Let--be-2} For every $t\geq0$, 
\[
\int_{0}^{t}\I_{\{\widetilde{W}_{s}\notin A\}}d\left(\int_{\mathbb{R}}L_{s}^{c}(\widetilde{W})\nu(dc)\right)=0
\]
almost surely.
\end{lem}

\begin{proof}
For every $c\in A$, the measure $dL_{s}^{c}(\widetilde{W})$ is supported
on the set 
\[
\{s\geq0:\widetilde{W}_{s}=c\}.
\]
Therefore, 
\[
\int_{0}^{t}\I_{\{\widetilde{W}_{s}\notin A\}}dL_{s}^{c}(\widetilde{W})=0,\quad c\in A.
\]
Since all integrands are nonnegative, Tonelli's theorem yields 
\begin{align*}
 & \int_{0}^{t}\I_{\{\widetilde{W}_{s}\notin A\}}d\left(\int_{\mathbb{R}}L_{s}^{c}(\widetilde{W})\nu(dc)\right)=\int_{\mathbb{R}}\int_{0}^{t}\I_{\{\widetilde{W}_{s}\notin A\}}dL_{s}^{c}(\widetilde{W})\nu(dc)=0.
\end{align*}
This completes the proof of the statement.
\end{proof}
We now prove weak existence.
\begin{proof}[Proof of weak existence]
 We first show that 
\begin{equation}
\frac{1}{\mu_{c}}L_{t}^{c}(Y)=\int_{0}^{t}\I_{\{Y_{s}=c\}}ds,\qquad c\in A.\label{eq:occupation-Y}
\end{equation}

Since $\widetilde{B}$ is a continuous martingale and $T$ is a continuous
time change, $\widetilde{B}_{T}$ is a continuous martingale with
respect to the time-changed filtration $\widetilde{\mathbb{F}}_{T}$.
Moreover, since $A$ has Lebesgue measure zero, the occupation-time
formula implies 
\[
\int_{0}^{T_{t}}\I_{\{\widetilde{W}_{s}\in A\}}ds=0.
\]
Consequently, 
\begin{align}
\widetilde{B}_{T_{t}} & =\int_{0}^{T_{t}}\I_{\{\widetilde{W}_{s}\notin A\}}d\widetilde{B}_{s}=\int_{0}^{t}\I_{\{Y_{s}\notin A\}}d\widetilde{B}_{T_{s}}.\label{eq:-6-1}
\end{align}

Since $\langle\widetilde{B}_{T}\rangle_{t}=T_{t}$ and $R_{T_{t}}=t$,
Lemma \ref{lem:Let--be-2} gives 
\begin{align}
T_{t} & =\int_{0}^{T_{t}}\I_{\{\widetilde{W}_{s}\notin A\}}ds=\int_{0}^{T_{t}}\I_{\{\widetilde{W}_{s}\notin A\}}dR_{s}=\int_{0}^{t}\I_{\{Y_{s}\notin A\}}ds.\label{eq:-7-1}
\end{align}

For $c\in A$, using the change of variables associated with the continuous
increasing function $R$, we obtain 
\begin{align}
\int_{0}^{t}\I_{\{Y_{s}=c\}}ds & =\int_{0}^{T_{t}}\I_{\{\widetilde{W}_{s}=c\}}dR_{s}=\int_{0}^{T_{t}}\I_{\{\widetilde{W}_{s}=c\}}d\left(\int_{\mathbb{R}}L_{s}^{x}(\widetilde{W})\nu(dx)\right).\label{eq:equality_for_int_of_I}
\end{align}
Here the contribution of $ds$ vanishes because 
\[
\int_{0}^{T_{t}}\I_{\{\widetilde{W}_{s}=c\}}ds=0.
\]
Since $dL_{s}^{x}(\widetilde{W})$ is supported on $\{\widetilde{W}_{s}=x\}$,
only the atom $x=c$ contributes to the last integral in (\ref{eq:equality_for_int_of_I}).
Hence 
\[
\int_{0}^{t}\I_{\{Y_{s}=c\}}ds=\frac{1}{\mu_{c}}L_{T_{t}}^{c}(\widetilde{W}).
\]
By the time-change property of local time, 
\[
L_{t}^{c}(Y)=L_{T_{t}}^{c}(\widetilde{W}),
\]
and therefore (\ref{eq:occupation-Y}) follows.

It remains to construct a Brownian motion $B$ such that 
\[
Y_{t}=x_{0}+\int_{0}^{t}\I_{\{Y_{s}\notin A\}}dB_{s}.
\]

Let $B^{0}$ be a Brownian motion independent of $\widetilde{B}$,
defined on an auxiliary probability space. Passing to the corresponding
product probability space if necessary, define 
\begin{equation}
B_{t}=\widetilde{B}_{T_{t}}+\int_{0}^{t}\I_{\{Y_{s}\in A\}}dB_{s}^{0}.\label{eq:-8-1}
\end{equation}
Then $B$ is a continuous local martingale. Since the two martingale
terms in (\ref{eq:-8-1}) are independent, their quadratic covariation
is zero. By (\ref{eq:-7-1}), 
\begin{align*}
\langle B\rangle_{t} & =T_{t}+\int_{0}^{t}\I_{\{Y_{s}\in A\}}\,ds=\int_{0}^{t}\I_{\{Y_{s}\notin A\}}\,ds+\int_{0}^{t}\I_{\{Y_{s}\in A\}}\,ds=t.
\end{align*}
Hence, by Lévy's characterization theorem, $B$ is a Brownian motion.

Finally, by (\ref{eq:-6-1}) and (\ref{eq:-8-1}), 
\begin{align*}
Y_{t} & =x_{0}+\widetilde{B}_{T_{t}}=x_{0}+\int_{0}^{t}\I_{\{Y_{s}\notin A\}}\,d\widetilde{B}_{T_{s}}\\
 & =x_{0}+\int_{0}^{t}\I_{\{Y_{s}\notin A\}}\,dB_{s}.
\end{align*}
Together with (\ref{eq:occupation-Y}), this shows that $(Y,B)$ is
a weak solution to (\ref{eq:-3}).
\end{proof}
We next prove uniqueness in law.

\subsection{Uniqueness in law}

We now prove uniqueness in law. As in Engelbert and Peskir \cite{key-7},
the main idea is to reverse the time change used in the existence
proof.

Let $(X,B)$ be an arbitrary weak solution of (\ref{eq:-3}). Define
\begin{equation}
\widehat{T}_{t}:=\int_{0}^{t}\I_{\{X_{s}\notin A\}}ds,\quad t\geq0.\label{eq:-27-1}
\end{equation}
By (\ref{eq:-3}), 
\[
X_{t}=x_{0}+\int_{0}^{t}\I_{\{X_{s}\notin A\}}dB_{s},
\]
and hence 
\begin{equation}
\langle X\rangle_{t}=\widehat{T}_{t}.\label{eq:-10-1-1}
\end{equation}

We first show that $\widehat{T}$ goes to infinity.
\begin{lem}
\label{lem:T_goes_to_infty}Let $\widehat{T}$ be defined by (\ref{eq:-27-1}).
Then 
\[
\widehat{T}_{t}\to\infty\quad\text{as }t\to\infty
\]
almost surely.
\end{lem}

\begin{proof}
Suppose, on the contrary, that the event 
\[
E:=\{\widehat{T}_{\infty}<\infty\}
\]
has positive probability.

Since $X-x_{0}$ is a continuous local martingale, the representation
theorem for continuous local martingales (see, e.g. \cite[Theorem V.1.7]{key-12})
implies that there exists a Brownian motion $W$ probably on an extended
probability space such that 
\begin{equation}
X_{t}=x_{0}+W_{\widehat{T}_{t}},\quad t\geq0,\label{eq:martingale-time-change}
\end{equation}
almost surely.

Fix $\omega\in E$ for which the representation (\ref{eq:martingale-time-change})
holds. Since $\widehat{T}_{\infty}(\omega)<\infty$ and the path of
$W$ is continuous, the set 
\[
K(\omega):=\left\{ x_{0}+W_{s}(\omega):0\leq s\leq\widehat{T}_{\infty}(\omega)\right\} 
\]
is a compact subset of $\mathbb{R}$. By (\ref{eq:martingale-time-change}),
\[
X_{t}(\omega)\in K(\omega)\qquad\text{for all }t\geq0.
\]

Since $A$ is countable, we have 
\[
\I_{\{X_{s}\in A\}}=\sum_{c\in A}\I_{\{X_{s}=c\}}.
\]
Therefore, using the local-time condition in (\ref{eq:-3}) and Tonelli's
theorem, 
\begin{align*}
t-\widehat{T}_{t} & =\int_{0}^{t}\I_{\{X_{s}\in A\}}ds=\sum_{c\in A}\int_{0}^{t}\I_{\{X_{s}=c\}}ds\\
 & =\sum_{c\in A}\frac{1}{\mu_{c}}L_{t}^{c}(X)=\int_{\mathbb{R}}L_{t}^{c}(X)\nu(dc).
\end{align*}

For the fixed $\omega\in E$, the local time $L_{t}^{c}(X)(\omega)$
vanishes for $c\notin K(\omega)$. Hence 
\begin{equation}
t-\widehat{T}_{t}(\omega)=\int_{K(\omega)}L_{t}^{c}(X)(\omega)\nu(dc).\label{eq:}
\end{equation}

By the time-change property of local time, for every fixed $c\in A$,
\begin{equation}
L_{t}^{c}(X)=L_{\widehat{T}_{t}}^{c-x_{0}}(W),\qquad t\geq0,\label{eq:time_change_equality}
\end{equation}
almost surely. Since $A$ is countable, there exists an event of probability
one on which (\ref{eq:time_change_equality}) holds simultaneously
for all $c\in A$ and all $t\geq0$. We may therefore assume that
the fixed $\omega\in E$ belongs to this event. 

Since $\widehat{T}_{t}(\omega)\leq\widehat{T}_{\infty}(\omega)<\infty$,
joint continuity of Brownian local time gives 
\[
C(\omega):=\sup_{\substack{0\leq s\leq\widehat{T}_{\infty}(\omega)}
}\sup_{y\in K(\omega)}L_{s}^{y-x_{0}}(W)(\omega)<\infty.
\]
Consequently, for every $c\in A\cap K(\omega)$ and $t\geq0$, 
\[
L_{t}^{c}(X)(\omega)=L_{\widehat{T}_{t}(\omega)}^{c-x_{0}}(W)(\omega)\leq C(\omega).
\]

Since $\nu$ is locally finite and $K(\omega)$ is compact, 
\[
\nu(K(\omega))<\infty.
\]
Therefore, using (\ref{eq:}), 
\begin{align*}
t-\widehat{T}_{t}(\omega) & =\int_{K(\omega)}L_{t}^{c}(X)(\omega)\nu(dc)\le C(\omega)\nu(K(\omega))<\infty.
\end{align*}
Hence 
\[
\sup_{t\geq0}\bigl(t-\widehat{T}_{t}(\omega)\bigr)<\infty.
\]

On the other hand, since $\widehat{T}_{t}(\omega)\leq\widehat{T}_{\infty}(\omega)<\infty$,
\[
t-\widehat{T}_{t}(\omega)\geq t-\widehat{T}_{\infty}(\omega)\to\infty\quad\text{as }t\to\infty,
\]
which is a contradiction. This completes the proof of the lemma.
\end{proof}
We now prove uniqueness in law.
\begin{proof}[Proof of uniqueness]
 By Lemma \ref{lem:T_goes_to_infty}, 
\[
\widehat{T}_{t}\to\infty\quad\text{as }t\to\infty
\]
almost surely. Therefore, its right-continuous inverse 
\[
\widehat{R}_{t}:=\inf\{s\geq0:\widehat{T}_{s}>t\},\qquad t\geq0,
\]
is finite for every $t\geq0$. The process $\widehat{R}$ is increasing
and right-continuous. Moreover, since $\widehat{T}$ is adapted and
continuous, $\widehat{R}_{t}$ is an $\mathbb{F}$-stopping time for
every $t\geq0$.

Consider the time-changed process 
\[
B_{t}^{1}:=X_{\widehat{R}_{t}}-x_{0},\quad t\geq0.
\]
We claim that 
\[
\langle B^{1}\rangle_{t}=t.
\]
Indeed, 
\begin{equation}
\langle B^{1}\rangle_{t}=\langle X\rangle_{\widehat{R}_{t}}=\widehat{T}_{\widehat{R}_{t}}=t.\label{eq:-11-1-1}
\end{equation}
Here, although $\widehat{T}$ need not be strictly increasing, it
is continuous. Hence, for its right-continuous inverse, 
\[
\widehat{T}_{\widehat{R}_{t}}=t.
\]
Moreover, if $\widehat{R}$ has a jump, then $\widehat{T}$ is constant
on the corresponding interval. Since 
\[
\langle X\rangle=\widehat{T},
\]
the continuous local martingale $X$ is constant on such intervals.
Consequently, $B^{1}=X_{\widehat{R}}-x_{0}$ is continuous.

It follows from Lévy's characterization theorem that $B^{1}$ is a
standard Brownian motion.

We next identify the inverse time change. By the definition of $\widehat{T}$
and the local-time condition in (\ref{eq:-3}), 
\begin{align*}
t & =\widehat{T}_{\widehat{R}_{t}}=\int_{0}^{\widehat{R}_{t}}\I_{\{X_{s}\notin A\}}ds=\widehat{R}_{t}-\int_{0}^{\widehat{R}_{t}}\I_{\{X_{s}\in A\}}\,ds\\
 & =\widehat{R}_{t}-\sum_{c\in A}\frac{1}{\mu_{c}}L_{\widehat{R}_{t}}^{c}(X).
\end{align*}
By the time-change property of local time, 
\[
L_{\widehat{R}_{t}}^{c}(X)=L_{t}^{c}(X_{\widehat{R}})=L_{t}^{c}(x_{0}+B^{1}),\quad c\in A.
\]
Since $A$ is countable, these equalities may be taken to hold simultaneously
for all $c\in A$ on an event of probability one. Therefore, 
\begin{equation}
\widehat{R}_{t}=t+\sum_{c\in A}\frac{1}{\mu_{c}}L_{t}^{c}(x_{0}+B^{1})=t+\int_{\mathbb{R}}L_{t}^{c}(x_{0}+B^{1})\nu(dc).\label{eq:R-uniqueness}
\end{equation}

In particular, $\widehat{R}$ is continuous and strictly increasing.
Hence $\widehat{T}$ is its inverse and 
\[
\widehat{R}_{\widehat{T}_{t}}=t,\quad t\geq0.
\]
Consequently, 
\[
X_{t}=X_{\widehat{R}_{\widehat{T}_{t}}}=x_{0}+B_{\widehat{T}_{t}}^{1}.
\]

The equality (\ref{eq:R-uniqueness}) shows that $\widehat{R}$ is
completely determined by the Brownian motion $B^{1}$ and $\widehat{T}$
is its inverse. Thus $X$ has the same law as the process obtained
from a standard Brownian motion by the time-change construction used
in the proof of weak existence. Therefore the law of $X$ is uniquely
determined. This proves uniqueness in law.
\end{proof}

\section{Identification with the Dirichlet-form construction}

\label{sec:dirichlet-identification}

We briefly recall the construction of sticky dynamics in \cite{key-6}
in dimension $d=1$. Let $A\subset\mathbb{R}$ be locally finite set,
i.e. $A\cap K$ is finite for every compact $K\subset\mathbb{R}$,
and let $\rho\in C^{1}(\mathbb{R})$ be strictly positive and bounded.
On $L^{2}(\mathbb{R},m_{\rho})$, where 
\[
m_{\rho}(dx)=\rho(x)\,dx+\sum_{c\in A}\rho(c)\delta_{c}(dx),
\]
consider the closure of the bilinear form 
\[
\mathcal{E}^{\rho}(f,g)=\frac{1}{2}\int_{\mathbb{R}}f'(x)g'(x)\rho(x)\,dx,\qquad f,g\in C_{c}^{\infty}(\mathbb{R}).
\]
We include the factor $1/2$ to match the normalization in (\ref{eq:-3}).
This corresponds to a deterministic time change that reduces the speed
of the process constructed in \cite[Theorem 5.2]{key-6} by a factor
of two.

The associated diffusion process $X$ satisfies, for every starting
point $x_{0}\in\mathbb{R}$ and a standard Brownian motion $B$ on
a possibly enlarged filtered probability space, 
\begin{equation}
X_{t}=x_{0}+\int_{0}^{t}\I_{\{X_{s}\notin A\}}\,dB_{s}+\frac{1}{2}\int_{0}^{t}\I_{\{X_{s}\notin A\}}(\log\rho)'(X_{s})\,ds,\quad t\ge0,\label{eq:dirichlet-sde}
\end{equation}
(see \cite[Proposition~3.10 and Theorems~5.2 and~5.20]{key-6}). The
following lemma identifies the case $\rho\equiv1$ with the process
constructed in this paper.
\begin{lem}
\label{lem:dirichlet-identification} Suppose that $\rho=1$. For
every $x_{0}\in\mathbb{R}$, the law of $X$, started at $x_{0}$,
on $C([0,\infty),\mathbb{R})$ coincides with the unique weak solution
of (\ref{eq:-3}) with $\mu_{c}=1$ for every $c\in A$. 
\end{lem}

\begin{proof}
When $\rho=1$, the drift in (\ref{eq:dirichlet-sde}) vanishes, so
$X$ satisfies the first equation of \ref{eq:-3}. By Theorem~\ref{thm:main},
it remains to verify the local-time condition at each $c\in A$.

By \cite[Theorems~4.9 and~5.16]{key-6}, after the same deterministic
time change, every $f\in C_{c}(\mathbb{R})\cap C_{b}^{2}(\mathbb{R}\setminus A)$
belongs to the generator domain and 
\[
f(X_{t})-f(X_{0})-\int_{0}^{t}\mathcal{L}f(X_{s})\,ds,\quad t\ge0,
\]
is a continuous local martingale for every starting point, where 
\begin{equation}
\mathcal{L}f(x)=\begin{cases}
\frac{1}{2}f''(x), & x\notin A,\\[2pt]
\frac{1}{2}\bigl(f'(x+)-f'(x-)\bigr), & x\in A.
\end{cases}\label{eq:dirichlet-generator}
\end{equation}
Here $C_{b}^{2}(\mathbb{R}\setminus A)$ means that the function and
its first two derivatives are continuous and bounded on $\mathbb{R}\setminus A$.

Fix $c\in A$. For an integer $n>\max\{|x_{0}|,|c|\}$, choose $\chi_{n}\in C_{c}^{\infty}(\mathbb{R})$
equal to one on a neighbourhood of $[-n,n]$, and set 
\[
f_{n,c}(x)=\chi_{n}(x)|x-c|,\qquad\tau_{n}=\inf\{t\geq0:|X_{t}|\geq n\}.
\]
The function $f_{n,c}$ belongs to the above class. On $(-n,n)$ its
second derivative vanishes away from $c$, and its first derivative
has a jump of size two at $c$ and no jump at any other point of $A$.
Thus $\mathcal{L}f_{n,c}=\I_{\{c\}}$ on $(-n,n)$, so 
\begin{equation}
|X_{t\wedge\tau_{n}}-c|-|x_{0}-c|-\int_{0}^{t\wedge\tau_{n}}\I_{\{X_{s}=c\}}\,ds,\quad t\ge0,\label{eq:dirichlet-kink-martingale}
\end{equation}
is a continuous local martingale. On the other hand, Tanaka's formula
gives 
\[
|X_{t\wedge\tau_{n}}-c|-|x_{0}-c|-L_{t\wedge\tau_{n}}^{c}(X)=\int_{0}^{t\wedge\tau_{n}}\mathrm{sgn}(X_{s}-c)\,dX_{s},\quad t\ge0,
\]
which is also a continuous local martingale by (\ref{eq:dirichlet-sde}).
Subtracting these two decompositions implies that 
\[
L_{t\wedge\tau_{n}}^{c}(X)-\int_{0}^{t\wedge\tau_{n}}\I_{\{X_{s}=c\}}\,ds,\quad t\ge0,
\]
is a continuous local martingale of finite variation starting at zero.
It therefore vanishes. Since $X$ is conservative and has continuous
paths, $\tau_{n}\uparrow\infty$ almost surely. Letting $n\to\infty$,
we obtain 
\begin{equation}
L_{t}^{c}(X)=\int_{0}^{t}\I_{\{X_{s}=c\}}\,ds,\quad t\geq0.\label{eq:dirichlet-local-time}
\end{equation}
As $A$ is countable, these identities hold simultaneously for all
$c\in A$ on an event of probability one. Equations (\ref{eq:dirichlet-sde})
and (\ref{eq:dirichlet-local-time}) implies that $X$ is a weak solution
of (\ref{eq:-3}) with $\mu_{c}=1$. The claimed equality in law follows
from Theorem~\ref{thm:main}. 
\end{proof}


\begin{thebibliography}{1}
\bibitem{key-12} D. Revuz and M. Yor. \textit{Continuous Martingales
and Brownian Motion}, Berlin: Springer, 1999.

\bibitem{key-7} Hans-Jürgen Engelbert and Goran Peskir. \textit{Stochastic
differential equations for sticky Brownian motion}, An International
Journal of Probability and Stochastic Processes, 2014.

\bibitem{key-11} K. Ito and H.P. McKean, Jr, \textit{Brownian motions
on a half line}, Illinois J. Math. 7, 1963.

\bibitem{key-6}Torben Fattler, Martin Grothaus, Nathalie Steil. \textit{Construction
of distorted Brownian motion with permeable sticky behaviour sets
with Lebesgue measure zero}, 2025, arXiv:2410.13814v2.

\bibitem{key-8} W. Feller. \textit{The parabolic differential equations
and the associated semi-groups of transformations}, Ann. of Math.
55, 1952.

\bibitem{key-9}W. Feller. \textit{Diffusion processes in one dimension},
Trans. Ann. Math. Soc. 77, 1954.

\bibitem{key-10}W. Feller. \textit{Generalized second order differential
operators and their lateral conditions}, Illinois J. Math. 1, 1957.

\end{thebibliography}
\end{document}